\RequirePackage{fix-cm}
\documentclass[smallextended]{svjour3}       
\smartqed  
\usepackage{graphicx}
\usepackage{amssymb}
\usepackage{mathrsfs}
\usepackage{mathtools}
\usepackage{enumerate}
\usepackage{amssymb}

\numberwithin{theorem}{section}

\numberwithin{equation}{section}

\numberwithin{lemma}{section}
\numberwithin{corollary}{section}
\numberwithin{example}{section}
\numberwithin{remark}{section}
\newtheorem{prob}{Problem}[section]
\numberwithin{prob}{section}
\newtheorem{assump}{Assumption}[section]
\numberwithin{assump}{section}

\begin{document}
	
	\title{Radii problems for Ma-Minda Starlikeness
		 \thanks{Kamaljeet thanks University Grant Commission, New-Delhi, India for providing Junior Research Fellowship under UGC-Ref. No.:1051/(CSIR-UGC NET JUNE 2017)
		}
	}
	
	\author{ K. Gangania \and
		S. Sivaprasad Kumar
	}
	
	\authorrunning{ Kamaljeet Gangania \and
		S. Sivaprasad Kumar
		} 
	
	\institute{S. Sivaprasad Kumar \at
		\email{spkumar@dce.ac.in}
		\and
         Kamaljeet Gangania \at
        \email{gangania.m1991@gmail.com}  \\         
		\and	
	 	Department of Applied Mathematics, Delhi Technological University,
		Delhi--110042, India 
	}
	
	\date{Received: date / Accepted: date}

	\maketitle
	
	\begin{abstract}
 For the standard Ma-Minda class $\mathcal{S}^{*}(\psi)$ of univalent starlike functions, we derive $\mathcal{S}^{*}(\psi)$-radii for some well-known special functions. In addition, we obtain the set of extremal functions for the classical problem
 $$\max_{f\in \mathcal{S}^{*}(\psi)}^{}\left|\Phi\left(\log{(f(z)/z)}\right)\right| \quad \text{or} \quad \max_{f\in \mathcal{S}^{*}(\psi)}^{}\Re\Phi\left(\log{(f(z)/z)}\right),$$ where $\Phi$ is a non-constant entire function. Moreover, we prove certain results on convolution and radius estimates for the case when $\psi(\mathbb{D})$ is starlike.
	\end{abstract}
	
	\section{Introduction}
\label{Sec:1}

Let $\mathcal{A}_{0}$ be the collection of analytic functions of the form $p(z)=1+\sum_{n=2}^{\infty}p_nz^n$ and $\mathcal{A}$ consists of analytic functions, $f$ normalized by the conditions $f(0)=0$ and $f'(0)=1$ defined in the unit disk $\mathbb{D}:=\mathbb{D}_1$, where $\mathbb{D}_r:=\{z: |z|<r\}$. The Carath\'{e}odory class, $\mathcal{P}$ consists of functions $p\in{\mathcal{A}}_0$ with $\Re{p(z)}>0$. The class of univalent functions in $\mathcal{A}$ is denoted by $\mathcal{S}$. In 1992, Ma and Minda~\cite{minda94} introduced and studied the following classes of starlike and convex univalent functions:
\begin{equation}\label{mindasclass}
\mathcal{S}^{*}(\psi):=\left\{f\in\mathcal{A}: \frac{zf'(z)}{f(z)}\prec\psi(z)\right\}
\end{equation}
and 
\begin{equation*}\label{mindakclass}
\mathcal{C}(\psi):=\left\{f\in\mathcal{A}: 1+\frac{zf''(z)}{f'(z)}\prec\psi(z)\right\},
\end{equation*}
where $\psi\in\mathcal{P}$ is univalent, $\psi(\mathbb{D})$ is symmetric with respect to real axis and starlike with respect to $1$. The symbol $ \prec $ denotes the usual subordination. Note that $\mathcal{S}^{*}(\tfrac{1+z}{1-z})$ reduces to the well-known class $\mathcal{S}^{*}$ of starlike functions. Today a good amount of literature exists for the different choices of $\psi$ in \eqref{mindasclass}. For example, one may see \cite{sinefun,virendraBell,Kumar-cardioid,mendi2exp,raina2015,sokol1996}. We also introduced and studied (see,\cite{Kumar-cardioid}) the class of cardioid starlike functions:

\begin{def}\cite{Kumar-cardioid}
	\begin{equation*}\label{class}
	\mathcal{S}^*_{\wp}:=\left\{f\in\mathcal{A}: \frac{zf'(z)}{f(z)}\prec\wp(z):=1+\mathbb{E}_{1,1} \right\}.
	\end{equation*}
\end{def}	
where $\mathbb{E}_{\alpha, \beta}$ be the normalized form of the Mittag-Leffler function, also see~\cite{BP-mitiig-2016}:
\begin{equation*}
\mathbb{E}_{\alpha, \beta}(z)=z+\sum_{n\geq2}\frac{\Gamma(\beta)}{\Gamma(\alpha(n-1)+\beta)}z^n, \quad (z, \alpha, \beta\in\mathbb{C}; \Re(\alpha) > 0, \beta\neq 0,-1,\cdots).
\end{equation*}

Several type of radius problems have been studied in \cite{Gandhi-2022,sinefun,virendraBell,Kumar-cardioid,mendi2exp,raina2015,sokol1996}. Let us now recall that
\begin{definition}
	For the subfamilies $\mathcal{G}_1$ and $\mathcal{G}_2$ of $\mathcal{A}$, we say that $r_0$ is the $\mathcal{G}_1$-radius of the class $\mathcal{G}_2$, if $r_0\in(0,1)$ is largest number such that $r^{-1}f(rz)\in \mathcal{G}_1$, $0<r\leq r_0$ for all $f\in \mathcal{G}_2$.
\end{definition}

Enormous interest in the radius problems regarding special functions started from the work of Brown~\cite{Brown-1960,Brown-1962}, Wilf~\cite{wilf-1962}, and Kreyszig and Todd~\cite{krzg-Tod-1960}.  Recently, the radii of starlikeness and convexity of some normalized special functions were studied widely for certain Ma-Minda sub-classes as they can be represented as Hadamard factorization under certain conditions. See the work on Bessel functions~\cite{abo-2018,szasz-Bessel-2015}, Struve functions~\cite{abo-2018,bdoy-2016}, Wright functions~\cite{btk-2-18}, Lommel functions~\cite{abo-2018,bdoy-2016} and  Legendre polynomials of odd degree~\cite{bulut-engel-2019}. We also refer to see \cite{bohra2018,Eker-Rabotnovfuc2022}. For more on recent radius problems, see \cite{ganga-CMFT,ganga-Mediter,ganga-iraBohr,ganga-Special,ganga-Contructive,ganga-GenBohr}.

\underline{\bf Bessel Function:} The Bessel function $\mathcal{J}_{\beta}$ of first kind of order $\beta\in\mathbb{C}$ is a particular solution of the homogeneous Bessel differential equation $$z^2w''(z)+zw'(z)+(z^2-{\beta}^2)w(z)=0$$ and have the following series expansion: 
\begin{equation*}
\mathcal{J}_{\beta}(z):= \sum_{n\geq1}\frac{(-1)^n}{n!\Gamma(n+\beta+1)}\left(\frac{z}{2}\right)^{2n+\beta},
\end{equation*}
where $z\in\mathbb{C}$ and $\beta\not\in \mathbb{Z}^{-}$. Let us consider the following three normalized functions expressed in terms of $\mathcal{J}_{\beta}(z)$
\begin{equation}\label{fb}
\left\{
\begin{array}
{lr}
f_{\beta}(z)=(2^{\beta}\Gamma(\beta+1)\mathcal{J}_{\beta}(z))^{1/\beta}=z-\frac{1}{4\beta(\beta+1)}z^3+\cdots, & \beta\neq0 \\
g_{\beta}(z)=2^{\beta}\Gamma(\beta+1)z^{1-\beta}\mathcal{J}_{\beta}(z)=z-\frac{1}{4(\beta+1)}z^3+\cdots,   &\\
h_{\beta}(z)=2^{\beta}\Gamma(\beta+1)z^{1-\beta/2}\mathcal{J}_{\beta}(\sqrt{z})=z-\frac{1}{4(\beta+1)}z^2+\cdots. &
\end{array}
\right.
\end{equation}
Since the zeros of $\mathcal{J}_{\beta}$ are real if $\beta>0$, therefore using the Weierstrass decomposition, we have for $\beta>0$:
\begin{equation*}
\mathcal{J}_{\beta}(z):=\frac{z^{\beta}}{2^{\beta}\Gamma(\beta+1)}\prod_{n\geq1}\left(1-\frac{z^2}{j^2_{\beta,n}}\right),
\end{equation*}
where $j_{\beta,n}$ is the $n$-th positive zero of $\mathcal{J}_{\beta}$ and satisfies $j_{\beta,n}<j_{\beta,n+1}$ for $n\in\mathbb{N}$. Thus we have
\begin{equation}\label{besel-rep}
\frac{z\mathcal{J}^{'}_{\beta}(z)}{\mathcal{J}_{\beta}(z)}=\beta-\sum_{n\geq1}\frac{2z^2}{j^2_{\beta,n}-z^2}.
\end{equation}

\underline{\bf Struve function:} The Struve function $\mathcal{\bf{H}}_{\beta}$ of first kind is a particular solution of the second-order inhomogeneous Bessel differential equation $$z^2w''(z)+zw'(z)+(z^2-{\beta}^2)w(z)=\frac{4(\frac{z}{2})^{\beta+1}}{\sqrt{\pi}\Gamma(\beta+\frac{1}{2})}$$ and have the following form:
\begin{equation*}
\mathcal{\bf{H}}_{\beta}(z):=\frac{(\frac{z}{2})^{\beta+1}}{\sqrt{\frac{\pi}{4}}\Gamma(\beta+\frac{1}{2})} {}_1 F_{2}\left(1;\frac{3}{2},\beta+\frac{3}{2};-\frac{z^2}{4}\right) ,
\end{equation*}
where $-\beta-\frac{3}{2}\notin\mathbb{N}$ and ${}_1 F_{2}$ is a hypergeometric function. Since it is not normalized, so we consider the following normalized functions involving $\mathcal{\bf{H}}_{\beta}$ :
\begin{equation}\label{nor-strv-uvw}
\left\{
\begin{array}
{lr}
U_{\beta}(z)=\left(\sqrt{\pi}2^{\beta}(\beta+\frac{3}{2}){\bf{H}}_{\beta}(z)\right)^{\frac{1}{\beta+1}}, &\\ V_{\beta}(z)=\sqrt{\pi}2^{\beta}z^{-\beta}\Gamma(\beta+\frac{3}{2}){\bf{H}}_{\beta}(z), & \\
W_{\beta}(z)=\sqrt{\pi}2^{\beta}z^{\frac{1-\beta}{2}}\Gamma(\beta+\frac{3}{2}){\bf{H}}_{\beta}(\sqrt{z}). &
\end{array}
\right.
\end{equation}
Moreover, for $|\beta|\leq\frac{1}{2}$, it has the Hadamard factorization given by
\begin{equation}\label{strv-facto}
{\bf{H}}_{\beta}(z)=\frac{z^{\beta+1}}{\sqrt{\pi}2^{\beta}\Gamma(\beta+\frac{3}{2})}\prod_{n\geq1}\left(1-\frac{z^2}{z^2_{\beta,n}}\right),
\end{equation}
where $z_{\beta,n}$ is the $n$-th positive root of ${\bf{H}}_{\beta}$ such that $z_{\beta,n+1}>z_{\beta,n}$ and $z_{\beta,1}>1$ and also from \eqref{strv-facto}, we obtain
\begin{equation}\label{strv-strlike}
\frac{z{\bf{H}}'_{\beta}(z)}{{\bf{H}}_{\beta}(z)}=(\beta+1)-\sum_{n\geq1}\frac{2z^2}{z^2_{\beta,n}-z^2}.
\end{equation}

\underline{\bf Lommel function:} The Lommel function $\mathcal{L}_{u,v}$ of first kind is a particular solution of the second-order inhomogeneous Bessel differential equation $$z^2w''(z)+zw'(z)+(z^2-{v}^2)w(z)=z^{u+1},$$
where $u\pm v\notin \mathbb{Z}^{-}$ and is given by
$$\mathcal{L}_{u,v}=\frac{z^{u+1}}{(u-v+1)(u+v+1)}{}_1F_2\left(1;\frac{u-v+3}{2},\frac{u+v+3}{2};-\frac{z^2}{4}\right),$$
where $\frac{1}{2}(-u\pm v-3)\notin \mathbb{N}$ and ${}_1 F_{2}$ is a hypergeometric function. Since it is not normalized, so we consider the following normalized functions involving $\mathcal{L}_{u,v}$ :
\begin{equation}\label{fL}
\left\{
\begin{array}
{lr}
f_{u,v}(z)=((u-v+1)(u+v+1)\mathcal{L}_{u,v}(z))^{\tfrac{1}{u+1}}, &\\
g_{u,v}(z)=(u-v+1)(u+v+1)z^{-u}\mathcal{L}_{u,v}(z), &\\
h_{u,v}(z)=(u-v+1)(u+v+1)z^{\frac{1-u}{2}}\mathcal{L}_{u,v}(\sqrt{z}). &
\end{array}
\right.
\end{equation}
Authors in \cite{abo-2018,bdoy-2016} obtained the radius of starlikeness for the following normalized functions expressed in terms of $\mathcal{L}_{u,v}$:
\begin{equation}\label{lomel-normalized}
f_{u-\tfrac{1}{2},\tfrac{1}{2}}(z),\quad g_{u-\tfrac{1}{2},\tfrac{1}{2}}(z) \quad\text{and}\quad h_{u-\tfrac{1}{2},\tfrac{1}{2}}(z),
\end{equation} 
where $0\neq u\in (-1,1)$.

\underline{\bf Legendre polynomial:} The Legendre polynomials $P_{n}$ are the solutions of the Legendre differential equation:
$$((1-z^2)P'_{n}(z))'+n(n+1)P_{n}(z)=0,$$
where $n\in \mathbb{Z}^{+}$ and using Rodrigues$'$ formula, $P_{n}$ can be represented in the form:
$$P_{n}(z)=\frac{1}{2^n n!}\frac{d^n(z^2-1)^n}{dz^n}$$
and it also satisfies the geometric condition $P_n(-z)=(-1)^n P_{n}(z)$. Moreover, the odd degree Legendre polynomials $P_{2n-1}(z)$  have only real roots which satisfy 
\begin{equation}\label{legdroot}
0=z_0<z_1<\cdots<z_{n-1}\quad\text{or}\quad -z_1>\cdots>-z_{n-1}.
\end{equation}
Thus, the normalized form is as follows:
\begin{equation}\label{legd1}
\mathcal{P}_{2n-1}(z):=\frac{P_{2n-1}(z)}{P'_{2n-1}(0)}=z+\sum_{k=2}^{2n-1}a_{k}z^{k}=a_{2n-1}z\prod_{k=1}^{n-1}(z^2-z^2_{k}).
\end{equation}

At this conjunction, motivated from the work \cite{abo-2018,bdoy-2016,btk-2-18,bulut-engel-2019,bohra2018,ganga-Special,szasz-Bessel-2015} it is natural to consider the radius problem :
\begin{prob}\label{Prob}
	Find the $\mathcal{S}^{*}(\psi)$-radii for the normalized functions given in \eqref{fb}, \eqref{nor-strv-uvw}, \eqref{lomel-normalized} and \eqref{legd1}.
\end{prob}

That is, $\mathcal{S}^*(\psi)$-radius and $\mathcal{C}(\psi)$-radius of $g \in \mathcal{A}$ is defined as follows:
\begin{equation*}
r_0(g) =\sup\{ r\in (0, r_0) : \frac{zg'(z)}{g(z)} \in \psi(\mathbb{D}), z\in \mathbb{D}_{r_0}  \}
\end{equation*}
and 
\begin{equation*}
r_0(g) =\sup\{ r\in (0, r_0) : 1+\frac{zg''(z)}{g'(z)} \in \psi(\mathbb{D}), z\in \mathbb{D}_{r_0}  \}.
\end{equation*}

Till date, for a specific given function $\psi$ the above problem was considered, see~\cite{bohra2018}. Certain special functions's radius of starlikeness of order $\alpha\in[0,1)$ is given in \cite{abo-2018,bdoy-2016,btk-2-18,bulut-engel-2019,szasz-Bessel-2015}. To solve this in general, we need to consider the following assumption:
\begin{assump}\label{Assump-diskLemma}
	Consider the Ma-Minda function $\psi$ as defined in \eqref{mindasclass}. Let $a \in \psi(\mathbb{D}) \cap \mathbb{R}$, $r_a$ is the radius depending on $a$ and assume the maximal disk $ |w-a|<r_a$ such that 
	$$\{w : |w-a|< r_a\} \subseteq \psi(\mathbb{D}).$$
\end{assump}
The following is an example of the Assumption~\ref{Assump-diskLemma}:
\begin{lemma}[\cite{Kumar-cardioid}]
	\label{disk_lem}
	Let $\wp(z)=1+ze^z$. Then  we have $\{w : |w-a|<R_a\} \subset \wp(\mathbb{D}),$
	where
	\[  R_a=
	\left\{
	\begin{array}
	{lr}
	(a-1)+{1}/{e}, &  1-{1}/{e}<a\leq1+(e-e^{-1})/{2}; \\
	e-(a-1),   & 1+(e-e^{-1})/{2}\leq a<1+e.
	\end{array}
	\right.
	\]
\end{lemma}

In the present investigation, we find the $\mathcal{S}^*(\psi)$-radii for the normalized Special functions given by \eqref{fb}, \eqref{nor-strv-uvw}, \eqref{lomel-normalized} and \eqref{legd1} using the Assumption~\ref{Assump-diskLemma}. Further, the generalization of a classical problem of maximization of Goluzin for the class $\mathcal{S}^*(\psi)$ is established. Various non-trivial radius problems using the concept of convolution (i.e, term by term multiplication between coefficients of two power series) are studied for the case of starlike domains $\psi(\mathbb{D})$ (for example $\wp(\mathbb{D})$) which also show the importance of radius of convexity. Moreover, we find the sufficient conditions for some normalized functions $f$ in $\mathcal{A}$ to be in $\mathcal{S}^*(\psi)$ in terms of it's coefficients.

\section{$\mathcal{S}^{*}(\psi)$-Radius For Some Special Functions }

Let us denote $\mathcal{S}^{*}(\psi)$-radius by $R[\mathcal{S}^{*}(\psi)]$.
In view of the Assumption~\ref{Assump-diskLemma}, we shall see that there exists an $r_1=\alpha\in (0,1]$ such that
$$R[\mathcal{S}^{*}(\psi)] \geq R[\mathcal{S}^{*}(1+\alpha z)].$$
The case of equality, i.e sharpness of radii can be observe when $\psi(-1)=1-\alpha$, which covers many classical cases. Applications of the following sequel of results on special functions can be seen for subclasses of starlike functions studied in \cite{sinefun,virendraBell,mendi2exp,raina2015,sokol1996}.

\begin{theorem}[Bessel function $\mathcal{J}_{\beta}$]\label{gbessel}
	Let $\beta>0$. Then the $\mathcal{S}^*(\psi)$-radii $r_{\psi}(f_{\beta})$, $r_{\psi}(g_{\beta})$ and $r_{\psi}(h_{\beta})$ of the functions $f_{\beta}$, $g_{\beta}$ and $h_{\beta}$ as given by \eqref{fb} are the smallest positive root of the following equations, respectively:
	\begin{enumerate}
		\item [$(i)$] $r\mathcal{J}'_{\beta}(r)-\beta(1-{r_1})\mathcal{J}_{\beta}(r)=0;$
		\item [$(ii)$] $r\mathcal{J}'_{\beta}(r)-(\beta-{r_1})\mathcal{J}_{\beta}(r)=0;$
		\item [$(iii)$] $\sqrt{r}\mathcal{J}'_{\beta}(\sqrt{r})-(\beta-2{r_1})\mathcal{J}_{\beta}(\sqrt{r})=0.$
	\end{enumerate}	
	The radii are sharp when $\psi(-1)=1-r_1$, where $r_1$ is radius of largest disk inside $\psi(\mathbb{D})$.
\end{theorem}

\begin{theorem}[Struve function $\mathcal{\bf{H}}_{\beta}$]\label{gstruve}
	Let $|\beta|\leq {1}/{2}$. Then the $\mathcal{S}^*(\psi)$-radii $r_{\psi}(U_{\beta})$, $r_{\psi}(V_{\beta})$ and $r_{\psi}(W_{\beta})$ of the functions $U_{\beta}$, $V_{\beta}$ and $W_{\beta}$ as given by \eqref{nor-strv-uvw} are the smallest positive root of the following equations, respectively:
	\begin{enumerate}
		\item [$(i)$]
		$r{\bf{H}}'_{\beta}(r)-(1-{r_1})(\beta+1){\bf{H}}_{\beta}(r)=0;$
		\item [$(ii)$]
		$r{\bf{H}}'_{\beta}(r)-((1+\beta)-{r_1}){\bf{H}}'_{\beta}(r)=0;$
		\item [$(iii)$] $\sqrt{r}{\bf{H}}'_{\beta}(\sqrt{r})-(1+\beta-2{r_1}){\bf{H}}_{\beta}(\sqrt{r})=0.$
	\end{enumerate}
	The radii are sharp when $\psi(-1)=1-r_1$, where $r_1$ is radius of largest disk inside $\psi(\mathbb{D})$.
\end{theorem}

For the convenience of notations, functions defined in \eqref{lomel-normalized} are written as $f_{u}, g_{u}$ and $h_{u}$, respectively. 
\begin{theorem}[Lommel function $\mathcal{L}_{u,v}$]\label{glommel}
	Let $0\neq u\in(-1,1)$ and write $\mathcal{L}_{u-\tfrac{1}{2}, \tfrac{1}{2}}(z)=:\mathcal{L}_{u}(z)$. Then the $\mathcal{S}^*(\psi)$-radii $r_{\psi}(f_{u})$, $r_{\psi}(g_{u})$ and $r_{\psi}(h_{u})$ of the functions $f_{u}$, $g_{u}$ and $h_{u}$ given by \eqref{lomel-normalized} are the smallest positive root of the following equations, respectively:
	\begin{enumerate}
		\item [$(i)$]
		$\left\{
		\begin{array}
		{ll}
		2r\mathcal{L}'_{u}(r)-(2u+1)(1-{r_1})\mathcal{L}_{u}(r)=0,     & for\quad u\in(-\frac{1}{2},1) \\
		2r\mathcal{L}'_{u}(r)-(2u+1)(1+{r_1})\mathcal{L}_{u}(r)=0, & for\quad u\in(-1,-\frac{1}{2});
		\end{array}
		\right.$
		
		\item [$(ii)$]
		$2r\mathcal{L}'_{u}(r)-(2u+1-2{r_1})\mathcal{L}_{u}(r)=0;$
		\item [$(iii)$] $2\sqrt{r}\mathcal{L}'_{u}(\sqrt{r})-(2u+1-4{r_1})\mathcal{L}_{u}(\sqrt{r})=0.$
	\end{enumerate}	
	The radii are sharp when $\psi(-1)=1-r_1$, where $r_1$ is radius of largest disk inside $\psi(\mathbb{D})$.
\end{theorem}
\begin{theorem}[Legendre polynomials $\mathcal{P}_{n}$]\label{glegendre}
	The $\mathcal{S}^*(\psi)$-radius $r_{\psi}(\mathcal{P}_{2n-1})\in (0,z_1)$ of the normalized odd degree Legendre polynomial is the smallest positive root of the following equation:
	\begin{equation*}
	r\mathcal{P}'_{2n-1}(r)-(1-{r_1})\mathcal{P}_{2n-1}(r)=0.
	\end{equation*}	
	The radii are sharp when $\psi(-1)=1-r_1$, where $r_1$ is radius of largest disk inside $\psi(\mathbb{D})$.	
\end{theorem}

The following result covers many celebrated and newly introduced classes:
\begin{corollary}\label{manyclass}
	Let $\alpha=r_1$ be the radius of the lagest disk $\{w: |w-1|< \alpha\}$ inside $\psi(\mathbb{D})$, where
	\begin{enumerate}
		\item [$(i)$]
		$\alpha=\min\left\{\left|1-\frac{1+A}{1+B}\right|, \left|1-\frac{1-A}{1-B}\right|\right\}=\frac{A-B}{1+|B|}$ when $\psi(z)= \frac{1+Az}{1+Bz}$, where $-1\leq B<A\leq1$;
		
		\item  [$(ii)$]
		$\alpha=\sqrt{2-2\sqrt{2}+\sqrt{-2+2\sqrt{2}}}$ when  $\psi(z)=\sqrt{2}-(\sqrt{2}-1)\sqrt{\frac{1-z}{1+2(\sqrt{2}-1)z}}$;
		
		\item  [$(iii)$]
		$\alpha=\sqrt{2}-1$ when $\psi(z)=\sqrt{1+z}$;
		
		\item [$(iv)$]
		$\alpha=e-1$ when $\psi(z)=e^z$;
		
		\item [$(v)$]
		$\alpha=2-\sqrt{2}$ when $\psi(z)=z+\sqrt{1+z^2}$;
		
		\item [$(vi)$]
		$\alpha=\frac{e-1}{e+1}$ when $\psi(z)=\frac{2}{1+e^{-z}}$;
		
		\item  [$(vii)$]
		$\alpha=\sin{1}$ when $\psi(z)=1+\sin{z}$;
		
		\item [$viii$]
		$\alpha=1-e^{e^{-1}-1}$ when $\psi(z)=e^{e^z -1}$;
		
		\item [$(ix)$]
		for the domains bounded by the conic sections (see \cite{kanas})
		$\Omega_\kappa:=\{w=u+iv: u^2>{\kappa}^2(u-1)^2+{\kappa}^2v^2; \kappa\in[0,\infty)  \},$ we have $$\alpha=\frac{1}{\kappa+1},$$ 
		where the boundary curve of $\Omega_\kappa$ for fixed $\kappa$ is represented by the imaginary axis $(\kappa=0)$, the right branch of a hyperbola $(0<\kappa<1)$, a parabola $(\kappa=1)$ and an ellipse $(\kappa>1)$. The univalent Carath\'{e}odory functions mapping $\mathbb{D}$ onto $\Omega_\kappa$ is given by
		\begin{equation*}
		\psi(z):=\psi_{\kappa}(z)= 
		\left\{
		\begin{array}{lll}
		\frac{1+z}{1-z} & $for$ & \kappa=0;\\
		1+\frac{2}{1-\kappa^2}\sinh^2(A(\kappa) arctanh\sqrt{z}) & $for$ & \kappa\in(0,1);\\
		1+\frac{2}{\pi^2}\log^2{\frac{1+\sqrt{z}}{1-\sqrt{z}}} & $for$ & \kappa=1;\\
		1+\frac{2}{\kappa^2-1}\sin^2\left(\frac{\pi}{2K(t)}F\left( \frac{\sqrt{z}}{\sqrt{t}}, t \right) \right) & $for$ & \kappa>1,		
		\end{array}	
		\right.
		\end{equation*}
		where $A(\kappa)=(2/\pi)\arccos(\kappa)$, $F(w,t)=\int_{0}^{w}\frac{dx}{\sqrt{(1-x^2)(1-t^2x^2)}}$ is the Legender elliptic integral of the first kind, $K(t)=F(1,t)$ and $t\in(0,1)$ is choosen such that $\kappa=\cosh(\pi K'(t)/2K(t))$.
	\end{enumerate}
	Then Theorem~\ref{gbessel}, Theorem~\ref{gstruve}, Theorem~\ref{glommel} and Theorem~\ref{glegendre} hold true for the class 
	$\mathcal{S}^{*}(\psi)$ for the above choices of $\psi$, respectively. The radii are sharp.
\end{corollary}

\begin{remark}
	Part $(i)$ of the Corollary~\ref{manyclass} generalizes several results given in \cite{abo-2018,bdoy-2016,bric-2014,btk-2-18,bulut-engel-2019,szasz-Bessel-2015} for radius of starlikeness of order $\eta\in [0,1)$.
\end{remark}	
\begin{remark}
	In the above corollary part $(iv)$, Theorem~\ref{gbessel} simplify the results \cite[Theorem~4.2, p.~119]{bohra2018}, \cite[Theorem~4.3, p.~120]{bohra2018} and \cite[Theorem~4.4, p.~122]{bohra2018}.
\end{remark}
\begin{remark}\label{Probelm-extremals}
	It is worth to mention that the function of the form $\psi(z)=1+\alpha z$ serve as an extremal to the Problem~\ref{Prob}, which indeed reduces a lot of calculation. 
\end{remark}	

We shall see that in view of the Remark~\ref{Probelm-extremals}, it is sufficient to study the above general cases with the help of the class $\mathcal{S}^*_{\wp}$ for the simplicity.

\begin{proof}[Proof of  Theorem \ref{gbessel}]
	Here we have $r_1=1/e$. We now prove the first part. Using the representation \eqref{besel-rep} and equation \eqref{fb}, we get
	\begin{equation}\label{fj-rel}
	\frac{zf'_{\beta}(z)}{f_{\beta}(z)}=\frac{z\mathcal{J}^{'}_{\beta}(z)}{\beta\mathcal{J}_{\beta}(z)}=1-\frac{1}{\beta}\sum_{n\geq1}\frac{2z^2}{j^2_{\beta,n}-z^2}.
	\end{equation} 
	Further using a result \cite[Lemma~3.2, p.~10]{ganga1997} and from \eqref{fj-rel}, $|z|=r<j_{\beta,1}$ we obtain 
	\begin{equation}\label{fdisk}
	\left|\frac{zf'_{\beta}(z)}{f_{\beta}(z)}-a\right| \leq \frac{2}{\beta}\sum_{n\geq1}\frac{j^2_{\beta,n}r^2}{j^4_{\beta,n}-r^4},
	\end{equation}
	where $a:=1-\frac{2}{\beta}\sum_{n\geq1}\dfrac{r^4}{j^4_{\beta,n}-r^4}$ and $j_{\beta,n}$ denotes the $n$-th positive zero of the Bessel function $\mathcal{J}_{\beta}$. Also a simple calculation shows that $a\leq1$. Thus for the disk~\eqref{fdisk} to lie inside $\wp(\mathbb{D})$, we need only to consider that $1-\tfrac{1}{e}<a<1+\tfrac{e-e^{-1}}{2},$ and so by Lemma~\ref{disk_lem}, we have
	$$\frac{2}{\beta}\sum_{n\geq1}\frac{j^2_{\beta,n}r^2}{j^4_{\beta,n}-r^4}\leq a-1+\frac{1}{e}=\frac{1}{e}-\frac{2}{\beta}\sum_{n\geq1}\frac{r^4}{j^4_{\beta,n}-r^4},$$
	or equivalently,
	\begin{equation}\label{bes-eq}
	\frac{2}{\beta}\sum_{n\geq1}\frac{r^2}{j^2_{\beta,n}-r^2}-\frac{1}{e}\leq0.
	\end{equation}
	Also using \eqref{fj-rel}, \eqref{bes-eq} can be written as $\dfrac{er\mathcal{J}'_{\beta}(r)}{\beta\mathcal{J}_{\beta}(r)}+1-e\geq0$. Note that in view of Lemma~\ref{disk_lem}, we can also obtain \eqref{bes-eq} directly from \eqref{fj-rel}. Moreover, in \eqref{bes-eq} we only need to replace $1/e$ by $r_1$ in view of Assumption~\ref{Assump-diskLemma} for a given $\psi$ for the general proof. Hence, without any loss of generality, we further proceed in general settings. Now let us consider the strictly decreasing continuous function
	$$\Psi(r):=\frac{1}{e}-\frac{2}{\beta}\sum_{n\geq1}\frac{r^2}{j^2_{\beta,n}-r^2}=r_1-\frac{2}{\beta}\sum_{n\geq1}\frac{r^2}{j^2_{\beta,n}-r^2}, \quad  r\in (0,j_{\beta,1}).$$
	Then $\lim_{r\rightarrow0}\Psi(r)=r_1=1/e>0$ and $\lim_{r\rightarrow j_{\beta,1}}\Psi(r)=-\infty$. Also $\Psi'(r)<0$, since $r<j_{\beta,1}$. So we may assume $r_{\psi}(f_{\beta})$ be the unique root of $\Psi(r)=0$ in $(0,j_{\beta,1})$ such that $f_{\beta}$ in $\mathcal{S}^*(\psi)$ in $|z|<r_{\psi}(f_{\beta})$.
	
	let us denote $r_{f_{\beta}}= r_{\psi}(f_{\beta})$. Then from \eqref{fj-rel} and \eqref{bes-eq}, we see that
	\begin{equation}\label{radius equality inequality}
	\frac{r_{f_{\beta}} f'_{\beta}(r_{f_{\beta}})}{f_{\beta}(r_{f_{\beta}})} =1-r_1,  
	\end{equation}
	where $r_1$ depends on $\psi(\mathbb{D})$ in view of Assumption~\ref{Assump-diskLemma} and hence,  $f_{\beta}$ belongs to $\mathcal{S}^{*}(1+ r_1 z)$ in $|z|<r_{f_{\beta}}$. 
	
	Now let $r_1\in (0,1]$ such that $w_{1}:=\{w: |w-1|< r_1\}$ is the maximal disk inside $\psi(\mathbb{D})$. 
	Let us write $$F_{f_{\beta}}(z)=\frac{zf'_{\beta}(z)}{f_{\beta}(z)}.$$ Since a function $f(z) \in \mathcal{S}^{*}(\psi)$ if and only if $e^{-it}f(e^{it}z) \in \mathcal{S}^{*}(\psi)$ for all $t\in \mathbb{R}$. Therefore, using \eqref{fj-rel} and \eqref{radius equality inequality} with $z=r_{f_{\beta}}$ along with $\psi(-1)=1-\alpha$, the maximality of the disk $w_{\alpha}$ implies that $F_{f_{\beta}}(|z|\leq r)$ do not lie inside $\psi(\mathbb{D})$ for $r\geq r_{f_{\beta}}$ for some suitable rotation of $f_{\beta}$. Hence, $f_{\beta}$ belongs to $\mathcal{S}^{*}(\psi)$ in $|z|<r_{f_{\beta}}$ and the radius $r_{f_{\beta}}$ is sharp.    
	
	Proof of other parts follows similarly. 
\end{proof}

\begin{proof}[ Proof of Theorem \ref{gstruve}]
	From \eqref{nor-strv-uvw} and \eqref{strv-strlike}, by logarithmic differentiation, we get 
	\begin{equation}\label{strv-main-eq}
	\left\{
	\begin{array}
	{lr}
	\frac{zU'_{\beta}(z)}{U_{\beta}(z)}=\frac{1}{\beta+1}\frac{z{\bf{H}}'_{\beta}(z)}{{\bf{H}}_{\beta}(z)}=1-\frac{1}{\beta+1}\sum_{n\geq1}\frac{2z^2}{z^2_{\beta,n}-z^2}, &\\
	\frac{zV'_{\beta}(z)}{V_{\beta}(z)}=-\beta+\frac{z{\bf{H}}'_{\beta}(z)}{{\bf{H}}_{\beta}(z)}=1-\sum_{n\geq1}\frac{2z^2}{z^2_{\beta,n}-z^2}, &\\
	\frac{zW'_{\beta}(z)}{W_{\beta}(z)}= \frac{1-\beta}{2}+\frac{\sqrt{z}{\bf{H}}'_{\beta}(\sqrt{z})}{{2\bf{H}}_{\beta}(\sqrt{z})}=1-\sum_{n\geq1}\frac{z}{z^2_{\beta,n}-z}. &
	\end{array}
	\right.
	\end{equation}
	
	Now applying the inequality $||x|-|y||\leq |x-y|$ and using the Assumption~\ref{Assump-diskLemma} in \eqref{strv-main-eq}, we see that  $U_{\beta}, V_{\beta}$ and $W_{\beta}$ belongs to $\mathcal{S}^*(\psi)$, respectively whenever
	\begin{equation}\label{strv-disk-eq}
	\left\{
	\begin{array}
	{lr}
	\left|\frac{zU'_{\beta}(z)}{U_{\beta}(z)}-1\right| \leq \frac{1}{\beta+1}\sum_{n\geq1}\frac{2r^2}{z^2_{\beta,n}-r^2}\leq r_1, &\\
	\left|\frac{zV'_{\beta}(z)}{V_{\beta}(z)}-1\right| \leq \sum_{n\geq1}\frac{2r^2}{z^2_{\beta,n}-r^2}\leq r_1, &\\
	\left|\frac{zW'_{\beta}(z)}{W_{\beta}(z)}-1\right| \leq \sum_{n\geq1}\frac{r}{z^2_{\beta,n}-r}\leq r_1 &
	\end{array}
	\right.
	\end{equation}
	holds, where $|z|=r<z_{\beta,1}$. Now to find the largest positive radius for which \eqref{strv-disk-eq} holds. Let us consider the strictly increasing continuous functions
	\begin{align*}
	\Psi_1(r):=\frac{1}{\beta+1}\sum_{n\geq1}\frac{2r^2}{z^2_{\beta,n}-r^2}-r_1, \quad
	\Psi_2(r):= \sum_{n\geq1}\frac{2r^2}{z^2_{\beta,n}-r^2}-r_1
	\end{align*}
	and
	\begin{equation*}
	\Psi_3(r):= \sum_{n\geq1}\frac{r}{z^2_{\beta,n}-r}-r_1 .
	\end{equation*}
	Since $\lim_{r\rightarrow0}\Psi_{i}(r)<0$, $\Psi'_{i}(r)>0$ for $i=1$ to $3$, $\lim_{r\rightarrow z_{\beta,1}}\Psi_{i}(r)>0$ for $i=1,2$ and  $\lim_{r\rightarrow {z^2_{\beta,1}}}\Psi_{3}(r)>0$, there exist the unique positive roots, $r_{\psi}(U_{\beta}), r_{\psi}(V_{\beta}) \in (0,z_{\beta,1})$ and $r_{\psi}(W_{\beta})\in(0,{z^2_{\beta,1}})$ for $\Psi_{i}$, respectively so that the inequalities in \eqref{strv-disk-eq} holds in $|z|<r_{\psi}(U_{\beta})$, $|z|<r_{\psi}(V_{\beta})$ and $|z|<r_{\psi}(W_{\beta})$, respectively. Further using \eqref{strv-main-eq} in $\Psi_i(r)=0$, respectively, we obtain the desired equations. Further, following the proof of Theorem~\ref{gbessel}, sharpness of the radii follows.   
\end{proof}	

\begin{proof}[Proof of Theorem \ref{glommel}]
	We prove the first part. Let $0\neq u\in(0,1)$. Then using a result from \cite{lommel-hadmrd} (also see \cite[Lemma 1, p.~3358]{bdoy-2016}), we can write the Lommel function $\mathcal{L}_{u-\tfrac{1}{2},\tfrac{1}{2}}$ as follows:
	\begin{equation}
	\mathcal{L}_{u-\tfrac{1}{2},\tfrac{1}{2}}(z)=\frac{z^{u+\tfrac{1}{2}}}{u(u+1)}{}_1F_2\left(1;\frac{u+2}{2},\frac{u+3}{2};-\frac{z^2}{4}\right)=\frac{z^{u+\tfrac{1}{2}}}{u(u+1)}\phi_{0}(z),
	\end{equation}
	where $$\phi_{0}(z)=\prod_{n\geq1}\left(1-\frac{z^2}{z^2_{u,0,n}}\right),$$
	and $z_{u,0,n}$ is the simple and real $n$-th positive root of $\phi_{0}$. Also $z_{u,0,n}\in (n\pi, (n+1)\pi)$ which ensures $z_{u,0,n}>z_{u,0,1}>\pi>1$. Now with this representation, after logarithmic differentiation, from \eqref{fL} we get
	\begin{equation*}
	\frac{zf'_{u}(z)}{f_{u}(z)}=\frac{z\mathcal{L}'_{u-\tfrac{1}{2},\tfrac{1}{2}}(z)}{(u+\tfrac{1}{2})\mathcal{L}_{u-\tfrac{1}{2},\tfrac{1}{2}}(z)}=1-\frac{1}{u+\tfrac{1}{2}}\sum_{n\geq1}\frac{2z^2}{z^2_{u,0,n}-z^2}.
	\end{equation*}
	Using the triangle inequality and the Assumption~\ref{Assump-diskLemma}, we have $f_{u}\in \mathcal{S}^*(\psi)$ provided 
	$$T(r):=\frac{1}{u+\tfrac{1}{2}}\sum_{n\geq1}\frac{2r^2}{z^2_{u,0,n}-r^2}-r_1\leq0$$
	holds for $|z|=r<z_{u,0,1}$, where $T(r)$ is a strictly increasing continuous function in $(0,z_{u,0,1})$. Since $\lim_{r\rightarrow0}T(r)<0$, $\lim_{r\rightarrow z_{u,0,1}}T(r)>0$ and $T'(r)>0$,  there exists a root $r_{\psi}(f_{u})\in(0,z_{u,0,1})$ so that $f_{u}\in \mathcal{S}^*(\psi)$ in $|z|<r_{\psi}(f_{u})$. Now for the case $u\in (-1,0)$, we proceed as in the case when $u\in(0,1)$, just replacing $u$ by $u+1$ and $\phi_{0}$ by $\phi_{1}$, where
	$$\phi_{1}(z)={}_1F_2\left(1;\frac{u+1}{2},\frac{u+2}{2};-\frac{z^2}{4}\right)=\prod_{n\geq1}\left(1-\frac{z^2}{z^2_{u,1,n}}\right)$$
	and $z_{u,1,n}$ be the $n$-th positive root of $\phi_1$.
	Proof for the part (ii) and (iii)  follows in a similar fashion as in the proof of Theorem~\ref{gstruve} by applying the Assumption~\ref{Assump-diskLemma} on the following two equations, respectively using $||x|-|y||\leq|x-y|$:
	\begin{equation*}
	\frac{zg'_{u}(z)}{g_{u}(z)}=-u+\frac{1}{2}+\frac{z\mathcal{L}'_{u-\tfrac{1}{2},\tfrac{1}{2}}(z)}{(u+\tfrac{1}{2})\mathcal{L}_{u-\tfrac{1}{2},\tfrac{1}{2}}(z)}=1-\sum_{n\geq1}\frac{2z^2}{z^2_{u,0,n}-z^2}
	\end{equation*}
	and 
	\begin{equation*}
	\frac{zh'_{u}(z)}{h_{u}(z)}=\frac{3-2u}{4}+\frac{\sqrt{z}\mathcal{L}'_{u-\tfrac{1}{2},\tfrac{1}{2}}(\sqrt{z})}{2\mathcal{L}_{u-\tfrac{1}{2},\tfrac{1}{2}}(\sqrt{z})}=1-\sum_{n\geq1}\frac{z}{z^2_{u,0,n}-z},
	\end{equation*}
	where $z_{u,0,n}$ is the $n$-th positive root of the function $\phi_0$. Further, following the proof of Theorem~\ref{gbessel}, sharpness of the radii follows.  
\end{proof}	

\begin{proof}[ Proof of Theorem \ref{glegendre}]
	From \eqref{legd1}, after logarithmic differentiation, we obtain
	\begin{equation}\label{legd2}
	\frac{z\mathcal{P}'_{2n-1}(z)}{\mathcal{P}_{2n-1}(z)}=1-\sum_{k=1}^{n-1}\frac{2z^2}{z^2_k-z^2}.
	\end{equation}
	Now applying Assumption~\ref{Assump-diskLemma} on \eqref{legd2}, we have $\mathcal{P}_{2n-1}\in \mathcal{S}^{*}(\psi)$ whenever
	\begin{equation}\label{legd3}
	\left|
	\frac{z\mathcal{P}'_{2n-1}(z)}{\mathcal{P}_{2n-1}(z)}-1\right|\leq \sum_{k=1}^{n-1}\frac{2r^2}{z^2_k-r^2}\leq r_1,
	\end{equation}
	where $|z|=r<z_1$ and $z_k$ satisfies the condition given in \eqref{legdroot}. Now let us consider the strictly increasing continuous function
	$$T(r):=\sum_{k=1}^{n-1}\frac{2r^2}{z^2_k-r^2}- r_1, \quad r\in(0,z_1).$$
	We have to show that $T(r)\leq0$ in $|z|\leq r<z_1$ so that \eqref{legd3} holds. Since $\lim_{r\rightarrow0}T(r)<0$, $\lim_{r\rightarrow z_1}T(r)>0$ and $T'(r)>0$, there exists a unique positive root $r_{\psi}(\mathcal{P}_{2n-1})\in (0,z_1)$ of $T(r)$ such that $\mathcal{P}_{2n-1}\in \mathcal{S}^*(\psi)$ in $|z|<r_{\psi}(\mathcal{P}_{2n-1})$. Further, following the proof of Theorem~\ref{gbessel}, sharpness of the radii follows.  
\end{proof}	

\section{An extremal problem for the class $\mathcal{S}^*(\psi)$: the region of variablity }
In 1961, Goluzin \cite{golu} obtained the set of extremal functions $f(z)=z/(1-xz)^2$, $|x|=1$ for the problem of maximization of the quantity 
$\Re\Phi\left(\log({f(z)}/{z})\right)$ or $\left|\Phi\left(\log({f(z)}/{z})\right)\right|$
over the class $\mathcal{S}^*$, where $\Phi$ is a non-constant entire function. In 1973, MacGregor \cite{T.H.Mac1973} proved the result for the class $\mathcal{S}^*(\alpha):=\{f\in\mathcal{A}: \Re(zf'(z)/f(z))>\alpha, \alpha\in[0,1)\}$. Later on Barnard~\cite{Barnard-1975} discussed this for Bounded starlike functions. Now, we present the result for  the Ma-Minda class:
\begin{theorem}\label{mc1973}
	Suppose $\Phi$ is a non-constant entire function and $0<|z_0|<1$ and assume that the class $\mathcal{S}^{*}(\psi)$ is closed. Then maximum of either
	\begin{equation}\label{functinal1-2}
	\Re\Phi\left(\log\frac{f(z_0)}{z_0}\right) \quad\text{or}\quad 	\left|\Phi\left(\log\frac{f(z_0)}{z_0}\right)\right|
	\end{equation}
	for functions in the class $\mathcal{S}^{*}(\psi)$ is attained only when the function is of the form 
	\begin{equation}
	f(z)=z\exp\int_{0}^{\zeta z}\frac{\psi(t)-1}{t}dt,
	\end{equation}
	where $|\zeta|=1.$
\end{theorem}
\begin{proof}
	Since the class $\mathcal{S}^{*}(\psi)$ is compact, therefore the problem under consideration has a solution. Moreover, in view of a result of Goluzin \cite{golu}, in \eqref{functinal1-2} it suffices to consider the continuous functional 
	\begin{equation*}
	\Re\Phi\left(\log\frac{f(z_0)}{z_0}\right).
	\end{equation*} 
	Let $f\in \mathcal{S}^{*}(\psi)$. Then using a result from \cite{minda94}, ${f(z)}/{z}\prec {f_0(z)}/{z}=:F(z),$
	where $	f_0(z)=z\exp\int_{0}^{z}\frac{\psi(t)-1}{t}dt$ or equivalenlty $\log(f(z)/z) \prec \log F(z)$. Thus, 
	\begin{equation*}
	g(z)=\Phi\left(\log\frac{f(z)}{z}\right) \prec \Phi(\log F(z))=G(z).
	\end{equation*}
	Note that $G$ is also non-constant as is $\Phi$. Thus for each $r\in (0,1)$ by subordination principle, we obtain $g({\mathbb{\overline{D}}_r}) \subset 	G({\mathbb{\overline{D}}_r})=\Omega.$
	Since $G(xz)\prec G(z)$ for $|x|\leq1$ is obvious, therefore for $|z_0|=r$, we have
	$\{g(z_0): g\prec G \;\text{in}\; \mathbb{D} \} = \Omega.$
	Now by considering a support line to the compact set $\Omega$, we conclude that
	\begin{equation*}
	\max_{f\in \mathcal{S}^{*}(\psi)} 
	\Re\Phi\left(\log\frac{f(z_0)}{z_0}\right)=\Re{w_1},\quad w_1\in\partial{\Omega}.
	\end{equation*}
	Since $G$ is also an open map, therefore there exists a point $z_1$ where $|z_1|=r$ and $G(z_1)=w_1$ such that among finitely many $w_1$, for one suitable $w_1$, we have
	\begin{equation*}
	\Phi\left(\log\frac{f(z_0)}{z_0}\right)=w_1,
	\end{equation*} 
	where $f$ is the solution for the extremal problem. Now by the well known Lindel\"{o}f Principle, we have 
	\begin{equation}\label{e3}
	\Phi\left(\log\frac{f(z)}{z}\right)= \Phi(\log{F(xz)}),
	\end{equation}
	that is, if $f$ is the desired solution, then \eqref{e3} holds for some $x$, $|x|=1$. Since $\Phi$ is non-constant analytic function, so we may write
	$$\Phi(w)=c_0+c_nw^n+c_{n+1}w^{n+1}+\cdots;\; c_n\neq0.$$
	If we set $\log(f(z)/z)=\alpha_1z+\alpha_2z^2+\cdots$ and $\log(F(z))=\beta_1z+\beta_2z^2+\cdots,$ then from \eqref{e3}, comparing the coefficients, we get $c_n\alpha^n_1=c_n\beta^n_1.$ Or equivalently, $\alpha^n_1=\beta^n_1$, which in particular implies that $|\alpha_1|=|\beta_1|$. Since $\log(f(z)/z)\prec \log{F(xz)}$, $|\alpha_1|=|\beta_1|$ is possible only if $\log(f(z)/z)= \log{F(xyz)}$ for some $|y|=1$. Therefore, we conclude that 
	$$f(z)=z\exp\int_{0}^{uz}\frac{\psi(t)-1}{t}dt,$$
	where $|u|=1$ if $f$ is a solution to the extremal problem. 
\end{proof}
\begin{remark}	
	Note that the analogous result for the class $\mathcal{C}(\psi)$ also holds. 
\end{remark}

Now as an application of the Theorem \ref{mc1973}, we obtain the result due to MacGregor \cite{T.H.Mac1973}:
\begin{corollary}\cite{T.H.Mac1973}
	Suppose $\Phi$ is a non-constant entire function and $0<|z_0|<1$. Then the maximum of the expression \eqref{functinal1-2} for functions in the class $\mathcal{S}^{*}(\alpha)$ is attained only when the function is of the form 
	\begin{equation*}
	f(z)=\frac{z}{(1-\zeta z)^{2-2\alpha}},\; |\zeta|=1.
	\end{equation*}
\end{corollary}
\begin{proof}
	If $f\in \mathcal{S}^*(\alpha)$, then $f(z)/z \prec 1/(1-z)^{2-2\alpha}$ and the result follows.  
\end{proof}

\begin{corollary}
	Suppose $\Phi$ is a non-constant entire function and $0<|z_0|<1$. Then the maximum of the expression \eqref{functinal1-2} for functions in the class $\mathcal{S}^{*}_{\wp}$ is attained only when the function is of the form 
	\begin{equation*}
	f(z)=z\exp(e^{\zeta z}-1),\; |\zeta|=1.
	\end{equation*}
\end{corollary}
\begin{proof}
	If $f\in \mathcal{S}^*_{\wp}$, then $f(z)/z \prec \exp{(e^z-1)}$ and the result follows. 
\end{proof}


\section{Convolution Radius: A Case Study For Starlike domains}
Note that if the function $\psi$ considered in the Ma-Minda class is a starlike function but not convex, then the following classical theorem doesn't hold.
\begin{theorem}\cite{minda94}\label{c1}
	Let $\psi(\mathbb{D})$ be convex, $g\in \mathcal{C}$ and $f\in\mathcal{S}^*(\psi)$. Then	$f*g\in \mathcal{S}^*(\psi).$
\end{theorem}

For instance, for $\psi(z)=1+z e^z$, $z+\sqrt{1+z^2}$, $e^{e^z-1}$ and $1+4z/3+2z^2/3$ Theorem~\ref{c1} is not valid. Therefore, we need to modify Theorem~\ref{c1} to accommodate such cases to further derive various radius problems. But first we need to recall an important result due to Ruscheweyh and Sheil-Small:
\begin{lemma}[\cite{rush-sheil-1973}, p.~126]\label{l1}
	Suppose that either	$g\in \mathcal{C}$, $h\in \mathcal{S}^*$ or else $g,h\in\mathcal{S}^*_{1/2}$. Then for any analytic function $G$ in $\mathbb{D}$, we have
	$$\frac{g*hG(z)}{g*h(z)}\in\overline{co}G(\mathbb{D}),$$
	where $\overline{co}G(\mathbb{D})$ is the closed convex hull of $G(\mathbb{D})$.
\end{lemma}

Keenly observing the proof of Lemma \ref{l1}, we see that the unit disk $\mathbb{D}$ can be replaced by the sub-disk $\mathbb{D}_{r} :=\{z: |z|<r \}$, where $0<r\leq1$ and consequently, we obtain the following modified result. Since the proof is similar, so it is omitted here.
\begin{lemma}\label{l11}
	Suppose either	$g\in \mathcal{C}$, $h\in \mathcal{S}^*$ or else $g,h\in\mathcal{S}^*_{1/2}$. Then for any analytic function $G$ in $\mathbb{D}_{r}$, we have
	$({g*hG(z)})/({g*h(z)})\in\overline{co}G(\mathbb{D}_{r})$, where $r\in[0,1]$.
\end{lemma}

This immediately gives the following fundamental result 
\begin{theorem}[Imrovement of Theorem~\ref{c1}]\label{c11}
	Let $r_0$ be the radius of convexity of $\psi$. If $g\in \mathcal{C}$ and $f\in\mathcal{S}^*(\psi)$. Then
	$f*g\in \mathcal{S}^*(\psi)$
	in $|z|<r$, where $r=\min\{r_0,1\}$.
\end{theorem}
\begin{corollary}
	Let $f\in\mathcal{S}^{*}_{\wp}$ and $g\in \mathcal{C}$. Then $f*g\in\mathcal{S}^{*}_{\wp}$ in $\mathbb{D}_{r_0}$, where $r_0=(3-\sqrt{5})/2$ is the radius of convexity of $\wp$.
\end{corollary}	

Now consider the operators $\mathcal{F}_i :\mathcal{A}\rightarrow \mathcal{A}$ defined by
$$\mathcal{F}_1(f)(z)=f*g_1(z)=zf'(z)$$
$$\mathcal{F}_2(f)(z)=f*g_2(z)=\frac{1}{2}(f(z)+zf'(z))$$
\begin{equation*}\label{operators}
\mathcal{F}_3(f)(z)=f*g_3(z)=\frac{k+1}{z^k}\int_{0}^{z}t^{k-1}f(t)dt,\quad \Re{k}>0,
\end{equation*}
where 
$g_3(z)=\sum_{n=1}^{\infty}{(k+1)}/{(k+n)}z^n$, $g_2(z)=(z-z^2/2)/(1-z^2)^2$ and $g_1(z)=z/(1-z)^2$. Note that the function $g_1$ is convex in $|z|<2-\sqrt{3}$, $g_2$ is convex in $|z|<1/2$ while $g_3\in \mathcal{C}$.
The above defined operators were introduced by Alexander, Livingston and Bernardi, respectively. Now we obtain the following result, where $\mathcal{S}^*_{SG}:=\mathcal{S}^*(\frac{2}{e^{-z}+1})$, $\mathcal{S}^*_{C}:=\mathcal{S}^*(1+4z/3+2z^2/3)$, $\mathcal{S}^*_{b}=\mathcal{S}^*(e^{e^z -1})$ and $\mathcal{S}^*_{s}:=\mathcal{S}^*(1+\sin{z})$ :

\begin{corollary} 
	Let $\mathcal{F}_i$, $i=1$ to $3$ be the operators as defined above.
	\begin{itemize}
		\item [$(i)$] Let $f\in \mathcal{S}^*_{\wp}$. Then $\mathcal{F}_i(f) \in \mathcal{S}^*_{\wp}$ in $\mathbb{D}_{r_i}$, where
		$r_1=2-\sqrt{3}$, $r_2=(3-\sqrt{5})/2$ and $r_3=(3-\sqrt{5})/2$.
		\item [$(ii)$] Let $f\in  \mathcal{S}^*_{C}$. Then $\mathcal{F}_i(f) \in \mathcal{S}^*_{C}$ in $\mathbb{D}_{r_i}$, where
		$r_1=2-\sqrt{3}$, $r_2=1/2$ and $r_3=1/2$. 
		\item [$(iii)$] Let $f\in \mathcal{S}^*_{s}$. Then $\mathcal{F}_i(f) \in \mathcal{S}^*_{s}$ in $\mathbb{D}_{r_i}$, where
		$r_1=2-\sqrt{3}$, $r_2=0.345$ and $r_3=0.345$.
		\item [$(iv)$] Let $f \in \mathcal{S}^*_{SG}$. Then $\mathcal{F}_i(f) \in \mathcal{S}^*_{SG}$ in $\mathbb{D}_{r_i}$, where
		$r_1=2-\sqrt{3}$, $r_2=1/2$ and $r_3=1$. 
	\end{itemize}
	The radii are sharp.
\end{corollary}

In 2010, Ali et al.~\cite{con2010} dealt with the problem of finding $\mathcal{S}^*(\psi)$-radii of the convolution $f*g$, between two starlike functions. In fact, they showed that if $f,g\in \mathcal{S}^*$ and $h_{\rho}(z)=f*g(\rho z)/\rho$, then $h_{\rho}\in\mathcal{SL}^*$ for $0\leq\rho\leq(\sqrt{5}-2)/(\sqrt{2}-1)\approx0.09778$. They used the property of the function $\psi$ being convex. Now using Theorem \ref{c11}, we can obtain the result even for the case when $\psi(\mathbb{D})$ is starlike. Here, we have shown the usability of the radius of convexity of $\psi$. 
\begin{theorem}\label{con-star}
	Let $f,g\in \mathcal{S}^*$ and $h_{\rho}(z):=f*g(\rho z)/\rho$. Then 
	\begin{itemize}
		\item [$(i)$]	$h_{\rho}\in \mathcal{S}^*_{\wp}$ for $0\leq\rho\leq (2e-\sqrt{4e^2-2e+1})/(2e-1)\approx0.0957$,
		
		\item [$(ii)$] $h_{\rho}\in \mathcal{S}^*_{C}$ for  $0\leq \rho \leq(3-\sqrt{7})/2\approx0.177124$,
		
		\item [$(iii)$] $h_{\rho}\in \mathcal{S}^*_{s}$ for $0\leq\rho\leq(\sqrt{{\sin1}^2+2\sin1+4}-2)/(2+\sin1)\approx0.185835$,
		
		\item [$(iv)$] $h_{\rho}\in\mathcal{S}^*_{b}$ for $0\leq\rho\leq(2e-\sqrt{3e^2+e^{2/e}})/(e+e^{1/e})\approx0.122919$,
		
		\item [$(v)$] $h_{\rho}\in\mathcal{S}^*_{SG}$ for $0\leq\rho\leq(\sqrt{7e^2+6e+3}-2(1+e))/(3e+1)\approx0.108309$.
		
	\end{itemize}
	The constants are best possible.
\end{theorem} 
\begin{proof}We only prove first part and rest part's proof also follow in a similar fashion. \\
	\label{disk}
	$(i):$ Let $H(z)=z+\sum_{n=2}^{\infty}n^2z^n=(z(1+z))/(1-z)^3$. It is easy to see that 
	\begin{equation}\label{thm-disk}
	\biggl|\frac{zH'(z)}{H(z)} -\frac{1+r^2}{1-r^2}\biggl| \leq \frac{4r}{1-r^2}, \quad |z|=r<1. 
	\end{equation}
	Now by Lemma \ref{disk_lem}, the disk \eqref{thm-disk}  lies inside the cardioid $\wp(\mathbb{D})$, provided 
	$$\frac{4r}{1-r^2} \leq \frac{1+r^2}{1-r^2}-1+\frac{1}{e} $$  
	which in turn gives $ r\leq r_0:= (2e-\sqrt{4e^2-2e+1})/(2e-1).$  Define the function $h:\mathbb{D} \rightarrow \mathbb{C}$  by $h(z):=f(z)*g(z)$. Then $h(z)=F(z)*G(z)*H(z)$, where $F$ and $G$ are, respectively defined as $zF'(z)=f(z)$  and $zG'(z)=g(z)$. Since $f, g\in \mathcal{S}^*$, it follows that $F*G\in \mathcal{C}$. Also, $H(r_0 z)/r_0 \in \mathcal{S}^*_{\wp}$. Hence, using Theorem \ref{c11}, we have
	$${F(z)*G(z)*H(\rho_0 z)}/{\rho_0} \in \mathcal{S}^*_{\wp},$$
	where  $\rho_0=\min\{r_0, r_c\}=r_0$ and $r_c=(3-\sqrt{5})/2$  is the radius of convexity of $\wp$. For $z=-\rho_0$, $zH'(z)/H(z)=(1+4z+z^2)/(1-z^2)=1-1/e$, which implies that $\rho_0$ is sharp. 
\end{proof}

\begin{remark}
	It is worthy to mention that in Theorem~\ref{con-star}, we need  $r_c$, radius of convexity. However, the sharp radius of convexity for the class $\mathcal{S}^*(\psi)$ is an open problem.
\end{remark}	

\begin{theorem}
	Let $f,g\in \mathcal{S}^*$ and $h_{\rho}(z):=f*g(\rho z)/\rho$. Then $h_{\rho}(z)\in \mathcal{S}^{*} \left(\frac{1+Az}{1+Bz} \right)$ for
	$$0\leq\rho\leq \frac{2(B^2-1)+\sqrt{4(1-B^2)^2+(A-B)^2}}{A-B}=:\rho_0,$$
	where $-1<B<A\leq1$.
\end{theorem}
\begin{proof}
	Since for the function $p(z)\prec {(1+Az)}/{(1+Bz)}$, we have 
	\begin{equation}\label{j-disk}
	\left|p(z)-\frac{1-AB}{1-B} \right|\leq \frac{A-B}{1-B^2}.
	\end{equation}
	Therefore, for the disk \eqref{thm-disk} to lie inside the disk \eqref{j-disk}, we must have
	$$\frac{1-AB}{1-B^2}- \frac{A-B}{1-B^2} \leq \frac{1+r^2}{1-r^2} \leq \frac{1-AB}{1-B^2}+\frac{A-B}{1-B^2}\; \text{
		and}\;
	\frac{4r}{1-r^2}\leq \frac{A-B}{1-B^2},$$
	which upon simplification hold for $r\leq r_0= \sqrt{(A-B)/(2+A+B)}$ and $r\leq\rho_0$ respectively, where $\rho_0$ is the smallest positive root of the following equation
	$$(A-B)r^2+4(1-B^2)r-(A-B)=0.$$
	Since $\min\{r_0,\rho_0\}=\rho_0$ and the class $\mathcal{S}^{*}\left(\frac{1+Az}{1+Bz}\right)$ is closed under convolution with convex functions, now the result follows in a similar way as in the part $(i)$ of Theorem~\ref{con-star}. 
\end{proof}

\section{Some Sufficient conditions for $\mathcal{S}^*(\psi)$ and Further Radius Results}
In this section, we determine the sufficient conditions for the functions $z/(1+\sum_{k=1}^{\infty}a_kz^k)$, $z/(1+z^n)^k$ and certain other types of functions to be in $\mathcal{S}^*(\psi)$. For the clarity, here we set $r_a=R_a$.
\begin{theorem}
	Let $f(z)={z}/({1+\sum_{k=1}^{\infty}a_kz^k})$. If the coefficients of $f$ satisfy
	$$|1-a|+\sum_{k=1}^{\infty}(R_a+|1-a-k|)|a_k|\leq R_a,$$
	where $a$ and $R_a$ are as defined in the Assumption~\ref{Assump-diskLemma}. Then $f\in \mathcal{S}^*(\psi)$.
\end{theorem}
\begin{proof}
	For $f(z)={z}/({1+\sum_{k=1}^{\infty}a_kz^k})$, we have
	$$\left|\frac{zf'(z)}{f(z)}-a\right|=\left|1-a-\frac{\sum_{k=1}^{\infty}ka_kz^k}{1+\sum_{k=1}^{\infty}a_kz^k}\right|.$$
	Thus by Assumption~\ref{Assump-diskLemma}, $f\in \mathcal{S}^*(\psi)$, if
	$\left|1-a-\frac{\sum_{k=1}^{\infty}ka_kz^k}{1+\sum_{k=1}^{\infty}a_kz^k}\right|\leq R_a.$
	The above inequality holds whenever
	$$|1-a|+\sum_{k=1}^{\infty}|1-a-k||a_k|r^k\leq R_a(1-\sum_{k=1}^{\infty}|a_k|r^k)$$
	or equivalently,
	$|1-a|+\sum_{k=1}^{\infty}(|1-a-k|+R_a)|a_k|r^k\leq R_a.$
	Letting $r$ tends to $1^{-}$, completes the proof. 
\end{proof}

\begin{theorem}
	Let $f(z)=z/(1+z^k)^n$, where $n,k\in \mathbb{Z}^{+}$ are fixed. Then $f\in \mathcal{S}^*(\psi)$ in
	$$|z|<\left(\frac{R_a-|1-a|}{R_a+|1-a-kn|}\right)^{1/k},$$
	where $a$ and $R_a$ are as defined in the Assumption~\ref{Assump-diskLemma}.
\end{theorem}
\begin{proof}
	For $f(z)=z/(1+z^k)^n$, we have
	$$\left|\frac{zf'(z)}{f(z)}-a\right|=\left|1-a\frac{knz^k}{1+z^k}\right|.$$
	Thus by Assumption~\ref{Assump-diskLemma}, $f\in \mathcal{S}^*(\psi)$, if
	$\left|1-a-\dfrac{knz^k}{1+z^k}\right|<R_a.$
	This inequality holds whenever
	$|1-a|+|1-a-kn||z|^k<R_a(1-|z|^k)$
	which upon simplification yields that
	$$|z|^k<\frac{R_a-|1-a|}{R_a+|1-a-kn|}.$$
	Hence the result follows.  
\end{proof}

\begin{theorem}
	Let $p(z)$ be a polynomial such that $p(0)=1$ and $\deg p(z)=m$. Let $R=\min\{|z| : p(z)=0, z\neq0\}$. Then the function
	$$f(z)=z(p(z))^{\beta/m} \in \mathcal{S}^*(\psi)$$ 
	in
	$$|z|<\frac{R(R_a-|1-a|)}{|\beta|+R_a-|1-a|},$$ 
	where $a$ and $R_a$ are as defined in the Assumption~\ref{Assump-diskLemma}.
\end{theorem}
\begin{proof}
	Assume that $z_k$, $(k=1,2,..., m)$ are zeros of the polynomial $p(z)$. For the function $f(z)=z(p(z))^{\beta/m}$, we have
	$$\frac{zf'(z)}{f(z)}=1+\frac{\beta}{m}\sum_{k=1}^{\infty}\frac{z}{z-z_k}$$
	or equivalently,
	$$\dfrac{zf'(z)}{f(z)}-a=1-a+\dfrac{\beta}{m}\sum_{k=1}^{\infty}\left(\dfrac{z}{z-z_k}+\dfrac{r^2}{R^2-r^2}-\dfrac{r^2}{R^2-r^2}\right).$$
	Thus by Assumption~\ref{Assump-diskLemma}, $f\in \mathcal{S}^*(\psi)$ whenever
	$$(|1-a|-R_a)(R^2-r^2)+|\beta|(Rr+r^2)<0,$$
	which is satisfied if $|z|=r<{R(R_a-|1-a|)}/{(|\beta|+R_a-|1-a|)}$.  
\end{proof}

Kuroki and Owa~\cite{kurokiOwa2011} introduced and studied the class $\mathcal{S}(\alpha,\beta)$ of functions $f$, which satisfy the condition $zf'(z)/f(z)\prec p_{\alpha,\beta}(z)$, where
\begin{equation*}\label{s(alpha,beta)}
p_{\alpha,\beta}(z):=1+\frac{\beta-\alpha}{\pi}i\log\frac{1-e^{2\pi i\frac{1-\alpha}{\beta-\alpha}}z}{1-z},
\end{equation*}
$\alpha<1$, $\beta>1$ and $p_{\alpha,\beta}$ maps $\mathbb{D}$ onto the convex domain $\{w\in\mathbb{C}: \alpha< \Re{w}<\beta\}$. Note that if $\alpha\ngeq0$ then this class also contains non-univalent functions, and univalent starlike if $1>\alpha\geq0$.

\begin{remark}
	Note that we can extend Theorem~\ref{1} ($\psi(z)\neq (1+z)/(1-z)$) and Theorem~\ref{2} for $\mathcal{S}^{*}(\psi)$-radius if we replace the radius $1/e$ by $r_1$, where $r_1$ is  given by the Assumption~\ref{Assump-diskLemma}.
\end{remark}

We now conclude this section some results explicitly for the class $\mathcal{S}^*_{\wp}$.

\begin{theorem}\label{1}
	Let $f\in \mathcal{S}(\alpha,\beta)$. Then $f\in \mathcal{S}^*_{\wp}$ in $\mathbb{D}_{r_0}$, where $r_0$ is the least positive root of the equation
	\begin{equation}\label{radius1}
	\frac{\beta-\alpha}{\pi}\left(\log\frac{1+\sqrt{2(1+\cos(2\pi\frac{1-\alpha}{\beta-\alpha}))}r+r^2}{1-r^2}+2\arctan\frac{r}{1-r}\right)-\frac{1}{e}=0.
	\end{equation}
\end{theorem}
\begin{proof}
	Consider the analytic function $p_{\alpha,\beta}(z):=1+\frac{\beta-\alpha}{\pi}i \log{q(z)},$
	where $$q(z)=\frac{1-{c}z}{1-z} \quad\text{and}\quad c=\exp\left(2\pi i\frac{1-\alpha}{\beta-\alpha}\right).$$
	Note that $q(z)$ is a bilinear transformation, maps $\mathbb{D}_{r}$ onto the disk:
	$$\left|q(z)-\frac{1+cr^2}{1-r^2}\right|\leq\frac{|1+c|r}{1-r^2},$$
	which implies  
	\begin{equation*}
	|q(z)| 
	\leq\frac{1+|1+c|r+r^2}{1-r^2},
	\end{equation*}
	and therefore,
	\begin{equation}\label{maxlog}
	\log|q(z)|\leq \log\left(\frac{1+|1+c|r+r^2}{1-r^2}\right).
	\end{equation}
	For any $\delta\in \mathbb{C}$ with $|\delta|=1$, we have $1+\delta z\prec 1+z$. So to maximize $|\arg(1+\delta z)|$, it suffices to consider $|\arg(1+z)|$. Now for $|z|=r$, we have
	\begin{equation}\label{maxarg}
	|\arg(1+z)|\leq\arctan\frac{r}{1-r}.
	\end{equation}
	Hence to apply Lemma \ref{disk_lem}, we need to maximize $|p_{\alpha,\beta}(z)-1|$, that is,
	\begin{equation}\label{pab}
	|p_{\alpha,\beta}-1|= \frac{\beta-\alpha}{\pi}\left|\log|q(z)|+i\arg\frac{1-cz}{1-z}\right|.
	\end{equation}
	Using \eqref{maxlog} and \eqref{maxarg} in \eqref{pab}, we see that
	\begin{equation*}
	|p_{\alpha,\beta}-1|\leq \frac{\beta-\alpha}{\pi}\left(\log\frac{1+|1+c|r+r^2}{1-r^2}+2\arctan\frac{r}{1-r}\right)\leq\frac{1}{e}
	\end{equation*}
	holds in $|z|<r_0$ whenever $r_0$ is the smallest positive root of \eqref{radius1}. 
\end{proof}	

Note that if we choose $\alpha=1+\frac{\delta-\pi}{2\sin{\delta}}$ and $\beta=1+\frac{\delta}{2\sin{\delta}}$, where $\pi/2\leq\delta<\pi$, then $\mathcal{S}(\alpha,\beta)$ reduces to the class $\mathcal{V}(\delta)$ introduced by Kargar et al. \cite{kargar-ebadian}.
\begin{corollary}
	Let $f\in \mathcal{V}(\delta)$. Then $f\in \mathcal{S}^*_{\wp}$ in $\mathbb{D}_{r_{\delta}}$, where $r_{\delta}$ is the least positive root of the equation
	\begin{equation*}
	\frac{1}{2\sin{\delta}}\left(\log\frac{1+\sqrt{2(1+\cos(2(\pi-\delta)))}r+r^2}{1-r^2}+2\arctan\frac{r}{1-r}\right)-\frac{1}{e}=0.
	\end{equation*}
\end{corollary} 

Now we consider the following class introduced in \cite{cho2019}:
\begin{equation}\label{s-lamda}
\mathcal{S}_{\lambda}:=\left\{f\in\mathcal{A} : \frac{f(z)}{z} \in P_{\lambda}\right\},
\end{equation}
where 
$P_{\lambda}:=\{p\in\mathcal{A}_{0}: \Re(e^{i\lambda}p(z))>0,\quad -{\pi}/{2} \leq\lambda\leq {\pi}/{2} \}$
denotes the class of tilted Carath\'{e}odory functions \cite{tilt}. Note that $P_0$ reduces to $\mathcal{P}$, the class of Carath\'{e}odory functions. For the function $p\in P_{\lambda}$, upper bound on the quantity $zp'(z)/p(z)$ is given by the following lemma that will be used for our next result:
\begin{lemma}\label{tilt-lem}\cite{tilt}
	If $p\in P_{\lambda}$, then
	$\left|{zp'(z)}/{p(z)}\right|\leq M(\lambda,r),$
	where
	\begin{equation*}
	M(\lambda,r)= 
	\left\{
	\begin{array}{lll}
	\frac{2r\cos{\lambda}}{r^2-2r|\sin{\lambda}|+1} & $for$ & r<|\tan\frac{\lambda}{2}|;\\
	\frac{2r}{1-r^2} & $for$ & r\geq |\tan\frac{\lambda}{2}|.	
	\end{array}	
	\right.
	\end{equation*}
	The equality holds for some point $z=re^{i\theta}$, $r\in(0,1)$ if and only if $p(z)=p_{\lambda}(yz)$, where $p_{\lambda}(z)=\frac{1+e^{-2i\lambda}z}{1-z}$ and  $y=e^{i(\theta_{0}-\theta)}$ with 
	\begin{equation*}
	\theta_{0}= 
	\left\{
	\begin{array}{lll}
	\frac{\pi}{2}+\lambda & $for$ & r<-\tan\frac{\lambda}{2};\\
	-\frac{\pi}{2}+\lambda  & $for$ & r< \tan\frac{\lambda}{2};\\
	\arcsin\left(\frac{1+r^2}{r^2-1}\right) +\lambda & $for$ &r\geq|\tan\frac{\lambda}{2}|.	
	\end{array}	
	\right.
	\end{equation*}
\end{lemma}

Next, we determine the largest radius $r$ such that the function $F(z):=f(z)g(z)/z\in \mathcal{S}^{*}_{\wp}$ in $|z|<r$, whenever $f,g\in \mathcal{S}_{\lambda}.$
\begin{theorem}\label{2}
	Let $c_{\lambda}=\cos{\lambda}$, $s_{\lambda}=\sin{\lambda}$ and $t_{\lambda}=|\tan({\lambda}/{2})|$. If $f,g\in \mathcal{S}_{\lambda}$, then $F\in\mathcal{S}^{*}_{\wp}$ in $\mathbb{D}_{r_0}$, where
	\begin{equation*}
	r_0:= 
	\left\{
	\begin{array}{lll}
	2ec_{\lambda}+|s_{\lambda}|+\sqrt{((4e^2-1)c_{\lambda}+4e|s_{\lambda}|)c_{\lambda}}, & $if$ & r<t_{\lambda};\\
	\sqrt{4e^2+1}-2e, & $if$ & r\geq t_{\lambda}.	
	\end{array}	
	\right.
	\end{equation*}
\end{theorem}
\begin{proof}
	Since $f,g\in \mathcal{S}_{\lambda}$, it follows that the functions $p(z)=f(z)/z$ and $q(z)=g(z)/z$ belong to the class $P_{\lambda}$ such that $F(z)=zp(z)q(z).$
	Thus
	\begin{equation*}
	\frac{zF'(z)}{F(z)}-1=\frac{zp'(z)}{p(z)}+\frac{zq'(z)}{q(z)}.
	\end{equation*}
	Now from Lemma \ref{tilt-lem}, we obtain
	\begin{equation*}
	\left|\frac{zF'(z)}{F(z)}-1\right|\leq 2M(\lambda,r).
	\end{equation*}
	Therefore, using Lemma \ref{disk_lem}, we conclude that if $2M(\lambda,r)\leq 1/e$, then $F\in\mathcal{S}^{*}_{\wp}$. Since $2M(\lambda,r)\leq 1/e$ holds whenever
	$\dfrac{2rc_{\lambda}}{r^2-2|s_{\lambda}|r+1} \leq \dfrac{1}{2e}$ if $r<t_{\lambda}$,
	and
	$\dfrac{2r}{1-r^2}\leq \dfrac{1}{2e}$ if $r\geq t_{\lambda}$;
	or equivalently
	\begin{equation*}
	r^2-2(|s_{\lambda}|+2ec_{\lambda})r+1\geq0,\quad \text{if}\quad r<t_{\lambda}
	\end{equation*}
	and
	\begin{equation*}
	r^2+4er-1\leq0,\quad \text{if}\quad r\geq t_{\lambda},
	\end{equation*}
	respectively. Hence the result follows with $r_0$ as given in the hypothesis. Further, for the functions
	$$f(z)=g(z)=\frac{z(1+e^{-2i\lambda}yz)}{1-yz},$$
	sharpness hold in view of Lemma \ref{tilt-lem}.  
\end{proof}


%
\section*{Conflict of interest}
The authors declare that they have no conflict of interest.



\end{document}